\documentclass[11pt,a4]{article}

\usepackage{subfigure}
\usepackage{graphicx}
\usepackage{sidecap}
\usepackage{amsmath, amsthm, amssymb}
\usepackage{array}
\usepackage{url}
\usepackage{psfrag}
\usepackage{color}

\usepackage[vcentering,dvips]{geometry}

\newtheorem{theorem}{Theorem}[section]
\newtheorem{lemma}[theorem]{Lemma}
\newtheorem{proposition}[theorem]{Proposition}
\newtheorem{corollary}[theorem]{Corollary}
\newtheorem{definition}[theorem]{Definition}

\newtheorem{remark}[theorem]{Remark}

\newcommand{\E}{{\mathbb{E}} \hspace{1mm}}

\newcommand{\R}{{\mathbb{R}}}
\newcommand{\N}{{\mathbb{N}}}
\newcommand{\C}{{\mathbb{C}}}

\begin{document}

\title{Sparse Legendre expansions via $\ell_1$-minimization}

\author{Holger Rauhut\thanks{H.\ Rauhut is with the Hausdorff Center for Mathematics and
the Institute for Numerical Simulation, University of Bonn, Endenicher Allee 60, 53115 Bonn, Germany, \rm{rauhut@hcm.uni-bonn.de}}, Rachel Ward \thanks{R.\ Ward is with the Courant
Institute of Mathematical Sciences, New York University, 251 Mercer Street, New York, NY, 10012, \rm{rward@cims.nyu.edu}}}
\date{March 1, 2010; revised \today}

\maketitle

\begin{abstract}
We consider the problem of recovering polynomials that are sparse with respect to the 
basis of Legendre polynomials from a small number of random samples.
In particular, we show that a Legendre $s$-sparse polynomial of maximal degree $N$ can be recovered from $m \asymp s \log^4(N)$ random samples that are chosen independently according
to the Chebyshev probability measure $d \nu(x) = \pi^{-1}(1-x^2)^{-1/2} dx$.  As an efficient
recovery method, $\ell_1$-minimization can be used. We establish these results by verifying
the restricted isometry property of a preconditioned random Legendre matrix.  We then extend these results to a large
class of orthogonal polynomial systems, including the Jacobi polynomials, of which the Legendre polynomials are a special case.
Finally, we transpose these results into the setting of approximate recovery for functions in
certain infinite-dimensional function spaces.

\end{abstract}

\noindent
{{\bf Key words:} Legendre polynomials, sparse recovery, compressive sensing,
$\ell_1$-minimization, condition numbers, random matrices, orthogonal polynomials.}

\medskip

\noindent
{\bf AMS Subject classification:} 
41A10, 
42C05, 
94A20, 
42A61, 
60B20, 
15A12, 
65F35, 
94A12. 

\section{Introduction}

Compressive sensing has triggered significant research activity in recent years. Its central motif is that sparse signals can be recovered from what was previously believed to be highly incomplete
information \cite{carota06,do06-2}.  In particular, it is now known \cite{carota06,cata06,ruve08,ra07,ra08,ra09-1} that an $s$-sparse trigonometric polynomial of maximal degree $N$ can be recovered from $m \asymp s \log^4(N)$ sampling points.  These $m$ samples  can be chosen as a random subset from  the discrete set $\{j/N\}_{j=0}^{N-1}$ \cite{carota06,cata06,ruve08}, or independently from the uniform measure on $[0,1]$, see \cite{ra07,ra08,ra09-1}.  

Until now, all sparse recovery results of this type required that the underlying basis be uniformly bounded like the trigonometric system, so as to be \emph{incoherent} with point samples \cite{caro07}.   As the main contribution of this paper, we show that this condition may be relaxed, obtaining comparable sparse recovery results for any basis that is bounded by a square-integrable envelope function.  As a special case, we focus on the Legendre system over the domain $[-1,1]$.    To account for the blow-up of the Legendre system near the endpoints of its domain, the random sampling points are drawn according to the Chebyshev probability measure.  This aligns with classical results on Lagrange interpolation which support the intuition that Chebyshev points are much better suited for the recovery of polynomials than uniform points are \cite{br97-1}.

In order to deduce our main results we establish the \emph{restricted isometry property} (RIP) for a preconditioned
version of the matrix whose entries are the Legendre polynomials evaluated at sample points chosen from the Chebyshev measure.  The concept of preconditioning seems to be new in the context
of compressive sensing, although it has appeared within the larger scope of sparse approximation in a different context in \cite{sv08}. It is likely that the idea of preconditioning can be exploited in other situations of interest as well.

Sparse expansions of multivariate polynomials in terms of tensor products
of Legendre polynomials recently appeared in the problem of numerically 
solving stochastic or parametric PDEs \cite{codesc10,alho10}. Our results indeed extend easily to tensor
products of Legendre polynomials, and the application of our techniques in this context 
of numerical solution of SPDEs seems very promising. 
Our results may also be transposed into the setting of function approximation.  In particular,  we show that the aforementioned sampling and reconstruction procedure is guaranteed to produce near-optimal approximations to functions
in infinite-dimensional spaces of functions having $\ell_p$-summable Fourier-Legendre
coefficients ($0<p < 1$), provided that
the maximal polynomial degree in the $\ell_1$-reconstruction procedure is fixed appropriately in terms of the sparsity level.

Our original motivation for this work was the recovery of sparse spherical 
harmonic expansions \cite{anasro99} from randomly located samples on the sphere. While our preliminary results in this context
seem to be only suboptimal \cite{rw-prep}, 
the results in the present paper apply at least to the recovery of functions on the
sphere that are invariant under rotations of the sphere around a fixed axis. 
Sparse spherical harmonic expansions were recently exploited with good numerical success
in the spherical inpainting problem for the cosmic microwave background \cite{abdefamongst08}, but 
so far this problem had lacked a theoretical understanding. 

We note that the Legendre polynomial transform has fast 
algorithms for matrix vector multiplication; see for instance \cite{dadeja97,herokose96,da03-1,postta98-1,ty08}.
This fact is of crucial importance in numerical algorithms used for reconstructing the original function from its sample values -- especially when the dimension
of the problem gets large.

Our results extend to any polynomial system which is orthogonal with respect to a finitely-supported weight function satisfying a mild continuity condition; this includes the Jacobi polynomials, of which the Legendre polynomials are a special case.   It turns out that the Chebyshev measure is universal for this rich class of orthogonal polynomials, in the
sense that our corresponding result requires the random sampling points to be drawn according
to the Chebyshev measure, independent of the particular weight function. 

\medskip

Our paper is structured as follows: Section 2 contains the main results for recovery of
Legendre-sparse polynomials. Section 3 illustrates these results with numerical experiments.
In Section 4 we recall known theorems on $\ell_1$-minimization and in Section 5, we prove the results presented in Section 2. Section 6 extends the results to a rich class of orthogonal polynomial systems, including the Jacobi polynomials, while Section 7 contains our main result on the recovery of continuous
functions that are well approximated by Legendre-sparse polynomials.

\paragraph{\bf Notation.}

Let us briefly introduce some helpful notation.  The $\ell_p$-norm on $\R^N$ is defined as 
$
\|z\|_p = \Big( \sum_{j=1}^N |z_j|^p \Big)^{1/p}, 1 \leq p < \infty,
$
and $\|z\|_\infty = \max_{j=1,\hdots,N} |z_j|$ as usual. 
The ``$\ell_0$-norm'', $\|z\|_0 = \#\{j: z_j \neq 0\}$,
counts the number of non-zero entries of $z$. A vector $z$ is called
$s$-sparse if $\|z\|_0 \leq s$, and the error of best $s$-term approximation
of a vector $z\in \R^N$ in $\ell_p$ is defined as
\[
\sigma_s(z)_p = \inf_{y : \|y\|_0 \leq s} \|y-z\|_p.
\]
Clearly, $\sigma_s(z)_p = 0$ if $z$ is $s$-sparse. Informally, $z$ is called compressible
if $\sigma_s(z)_1$ decays quickly as $s$ increases. A result due to Stechkin, see e.g.\ \cite[Lemma 3.1]{fora10-1}, states that, for 
$q<p$, 
\begin{equation}\label{Stechkin}
\sigma_s(x)_p \leq s^{1/p-1/q} \|x\|_q;
\end{equation}
thus, vectors $x \in B_q^N = \{x \in \R^N, \|x\|_q \leq 1\}$ for $0<q\leq 1$ 
can be considered a good model for
compressible signals.

For $N \in \N$, we use the notation $[N] = \{1,\hdots,N\}$.  In this article, $C>0$ will always denote a universal constant that
might be different in each occurence. 

The Chebyshev probability measure (also referred to as arcsine distribution) on $[-1,1]$ is given by  $d\nu(x) = \pi^{-1} (1 - x^2)^{-1/2}dx$.  
If a random variable $X$ is uniformly distributed on $[0,\pi]$, then the random variable $Y = \cos{X}$ is distributed according to the Chebyshev measure.

\section{Recovery of Legendre-sparse polynomials from a few samples}

Consider the problem of recovering a polynomial $g$ 
from $m$ sample values $g(x_1),\hdots,g(x_m)$.  If the number of sampling points is less than or 
equal to the degree of $g$, such reconstruction is impossible in general due to 
dimension reasons. Therefore, as usual in the compressive sensing literature, 
we make a sparsity assumption. 
In order to introduce a suitable
notion of sparsity we consider the basis of Legendre polynomials $L_n$ on $[-1,1]$, normalized so as to be orthonormal with respect to the uniform measure on $[-1,1]$, i.e. $\frac{1}{2} \int_{-1}^{1} L_n(x) L_{\ell}(x) dx = \delta_{n,\ell}$.  




An arbitrary real-valued 
polynomial $g$ of degree $N-1$ can be expanded in terms of Legendre polynomials
\begin{equation}\label{gLegendre}
g(x) = \sum_{n=0}^{N-1} c_n L_n(x), \quad x \in [-1,1].
\end{equation}
If the coefficient vector $c \in \R^{N}$ is $s$-sparse, we call the corresponding polynomial
\emph{Legendre $s$-sparse}, or simply Legendre-sparse. If $\sigma_s(c)_1$ decays 
quickly as $s$ increases, then $g$ is called Legendre--compressible.

We aim to reconstruct Legendre--sparse polynomials, and more generally Legendre--compressible polynomials, of maximum degree $N-1$ from $m$ samples $g(x_1),\hdots,g(x_m)$, where $m$ 
is desired to be small -- at least smaller than $N$. Writing $g$ in the form \eqref{gLegendre}
this task clearly amounts to reconstructing the coefficient vector $c \in \R^{N}$. 

To the set of $m$ sample points $(x_1, \hspace{1mm} \hdots \hspace{1mm}, x_m)$ we associate the $m \times N$ \emph{Legendre} matrix $\Phi$ defined component-wise by 
\begin{equation}
\label{legendre_matrix}
\Phi_{j,k} = 
L_{k-1}(x_j), \quad j \in [m],\hspace{2mm}  k \in [N].
\end{equation}
Note that the samples $y_j = g(x_j)$ may be expressed concisely in terms of the coefficient 
vector $c \in \mathbb{R}^{N}$ according to
$$
y = 
\Phi c.
$$
Reconstructing $c$ from the vector $y$ amounts to solving this system of linear
equations. As we are interested in the underdetermined case $m < N$, this system
typically has infinitely many solutions, and our task is to single out the original sparse $c$. The obvious but naive approach for doing this is by solving for the sparsest solution that agrees with the measurements, 
\begin{equation}\label{l0:prog}
\min_{ z \in \R^{N}}  \|z\|_0 \quad \mbox{subject to}\quad \Phi z = y.
\end{equation}
Unfortunately, this problem is NP-hard in general \cite{avdama97,aw10}. To overcome this computational
bottleneck the compressive sensing
literature has suggested various tractable alternatives \cite{gitr07,carota06,netr08}, 
most notably $\ell_1$-minimization (basis pursuit) \cite{chdosa99,carota06,do06-2}, 
on which we focus in this paper. Nevertheless, it follows from our findings that 
greedy algorithms such as CoSaMP \cite{netr08} or Iterative Hard Thresholding \cite{blda09} may also be
used for reconstruction.

\medskip

Our main result is that any Legendre $s$-sparse polynomial may be recovered efficiently from a number of samples $m \asymp s \log^3(s)\log(N)$.  Note that at least up to the logarithmic factors, this rate is optimal. Also the 
condition on $m$ is implied by the simpler one $m \asymp s \log^4{N}$
Reconstruction is also robust: \emph{any} polynomial may be recovered efficiently to within a factor of its best approximation by a Legendre $s$-sparse polynomial, and, if the measurements are corrupted by noise, $g(x_1) + \eta_1, \hdots, g(x_m) + \eta_m$, to within an additional factor of the noise level $\varepsilon = \| \eta \|_{\infty}$.  We have

\begin{theorem}
\label{uniform:noise}
Let $N,m,s \in \N$ be given such that
$$
m \geq C s \log^3(s) \log(N).
$$
Suppose that $m$ sampling points $(x_1, \hdots, x_m)$ are drawn independently at random from the Chebyshev  
measure, and consider the $m \times N$ Legendre matrix $\Phi$ with entries $\Phi_{j,k} = L_{k-1}(x_j)$, and the $m \times m$ diagonal matrix ${\cal A}$ with entries $a_{j,j} = (\pi/2)^{1/2}(1-x_j^2)^{1/4}$.   Then with probability exceeding
$1-N^{-\gamma \log^3(s)}$ the following 
holds for all polynomials $g(x) = \sum_{k=0}^{N-1} c_k L_k(x)$.
Suppose that noisy sample values $y = \big( g(x_1) + \eta_1, \hdots, g(x_m) + \eta_m \big) = \Phi c + \eta$ are observed, and $\|{\cal A}\eta\|_{\infty} \leq \varepsilon$.   Then the coefficient vector $c = (c_0, c_1, \hdots, c_{N-1})$ is recoverable to within a factor of its best $s$-term approximation error  
and to a factor of the noise level by solving the inequality-constrained $\ell_1$-minimization problem
\begin{align}\label{relaxed}
c^{\#} = \arg \min_{z \in \R^N} \| z \|_1 \quad \mbox{ subject to } \quad \| {\cal A} \Phi z - {\cal A} y \|_2 \leq \sqrt{m}\varepsilon.
\end{align}
Precisely, 
\begin{equation}\label{l1:approx2}
\| c -c^{\#} \|_{2} \leq \frac{C_1 \sigma_s(c)_1}{\sqrt{s}} + C_2\varepsilon,
\end{equation}

and
\begin{equation}\label{l1:approx}
\|c-c^{\#}\|_{1} \leq D_1 \sigma_s(c)_1 + D_2 \sqrt{s} \varepsilon.
\end{equation}
The constants $C, C_1,C_2,D_1,D_2$, and $\gamma$ are universal.
\end{theorem}

\begin{remark}\label{rem22}
\begin{itemize}
\item[(a)] \emph{In the noiseless ($\varepsilon = 0$) and exactly $s$-sparse case ($\sigma_s(x)_1 = 0$), the above theorem
implies exact recovery via} 
\[
c^{\#} = \arg \min_{z \in \R^N} \| z \|_1 \quad \mbox{ \emph{subject to} } \quad \Phi z = y.
\]
\item[(b)] \emph{The condition $\|{\cal A}\eta\|_{\infty} \leq \varepsilon$ is satisfied in particular if $\|\eta\|_\infty \leq \varepsilon$.}

\item[(c)]
\emph{The proposed recovery method \eqref{relaxed} is \emph{noise-aware}, in that it requires knowledge of the noise level $\varepsilon$ a priori.
One may remove this drawback by using other reconstruction algorithms such as 
CoSaMP \cite{netr08} or Iterative Hard Thresholding \cite{blda09} which also achieve the reconstruction rates \eqref{l1:approx2} and \eqref{l1:approx} under the stated hypotheses, but do not require knowledge of $\varepsilon$ \cite{blda09,netr08}.  Actually, those algorithms always return $2s$-sparse vectors as approximations, in which case the $\ell_1$-stability result \eqref{l1:approx} follows immediately from \eqref{l1:approx2}, see \cite[p.\ 87]{b09} for details.   }

\end{itemize}
\end{remark}

\section{Numerical Experiments}

Let us illustrate the results of Theorem \ref{uniform:noise}.  In Figure $1(a)$ we plot a polynomial $g$ that is $5$-sparse in Legendre basis and with maximal degree $N = 80$ along with $m = 20$ sampling points drawn independently from the Chebyshev measure.  This polynomial is reconstructed exactly from the illustrated sampling points as the solution to the $\ell_1$-minimization problem \eqref{relaxed} with $\varepsilon = 0$.   In Figure $1(b)$ we plot the same 
Legendre-sparse polynomial in solid line, but the $20$ samples 
have now been corrupted by zero-mean Gaussian noise $y_j = g(x_j) + \eta_j$.  Specifically, we take $\E (|\eta_j|^2 )  = 0.025$, so that the expected noise level $\varepsilon \approx 0.16$.    In the same figure, we superimpose 
in dashed line the polynomial obtained from these noisy measurements as the solution 
of the inequality-constrained $\ell_1$-minimization problem \eqref{relaxed} with noise level 
$\varepsilon  = 0.16$.  

\begin{figure}[h]
\label{fig:illustrate}
\subfigure{
\includegraphics[width=0.45\textwidth]{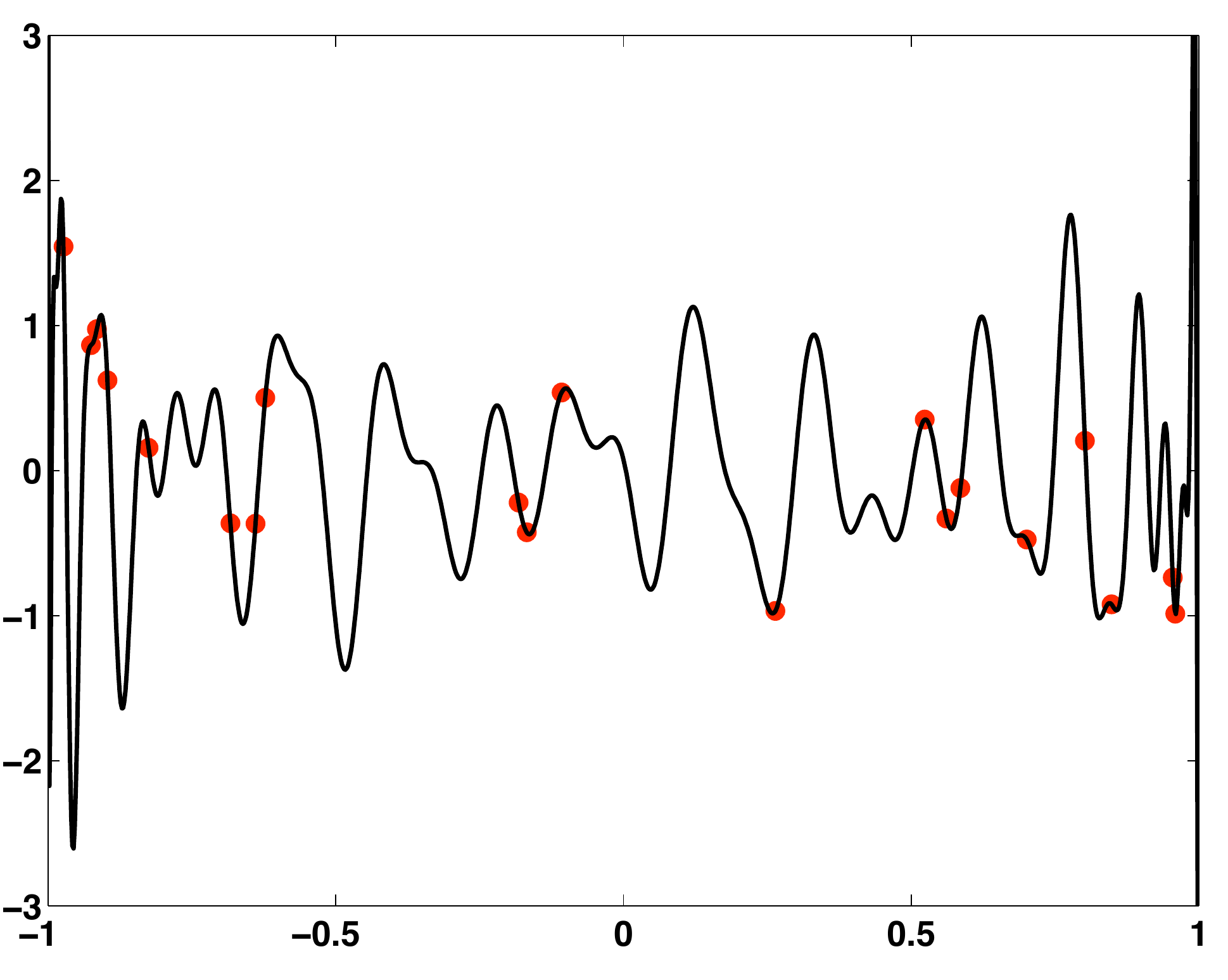}
}\hfill
\subfigure{
\includegraphics[width=0.45\textwidth]{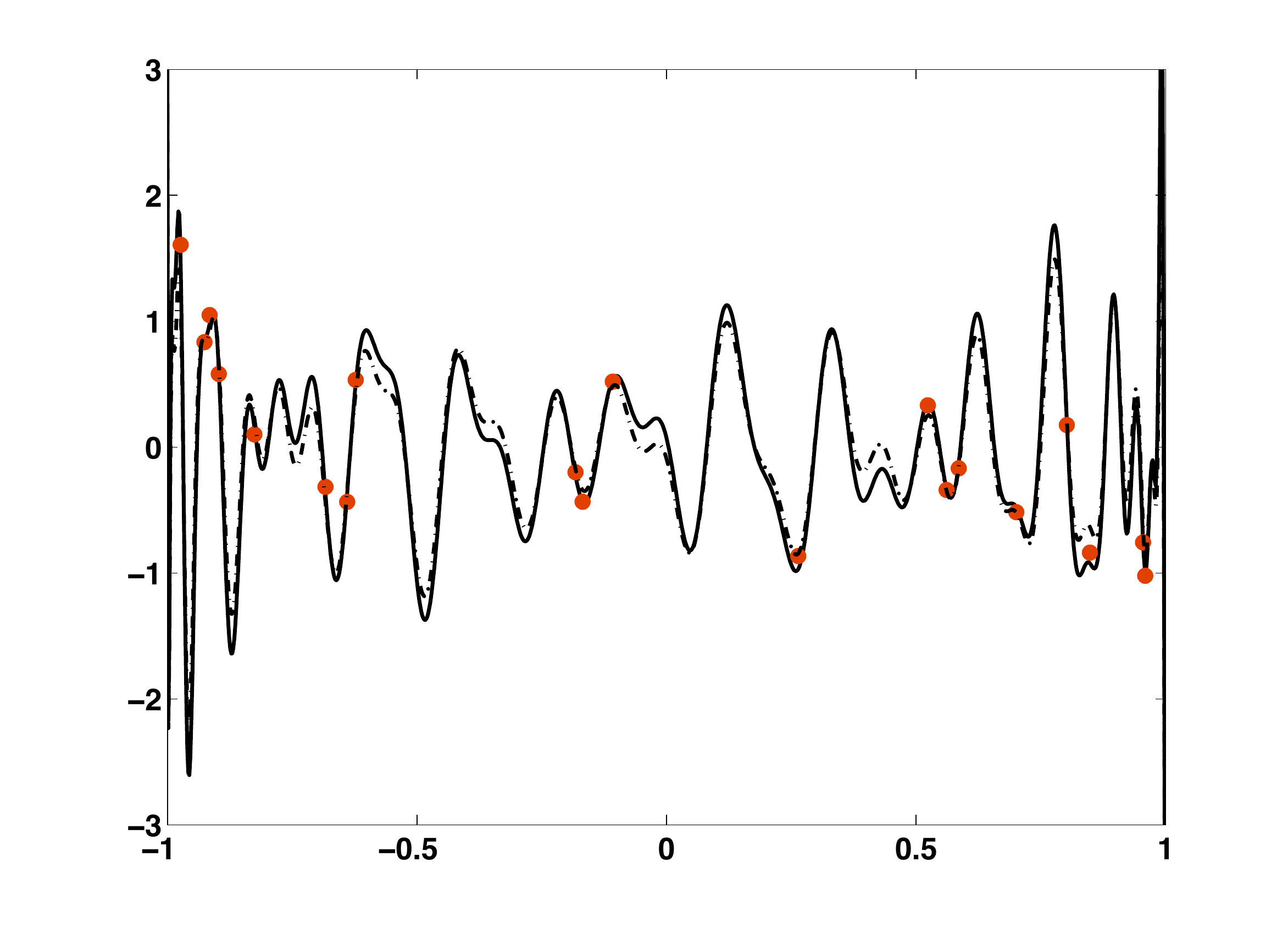}
}
\caption{(a) A Legendre-$5$-sparse polynomial of maximal degree $N = 80$, and its exact reconstruction from $20$ samples drawn independently from the Chebyshev distribution.  (b) The same polynomial (solid line), and its approximate reconstruction from $20$ samples corrupted with noise (dashed line).}
\end{figure}


To be more complete, we plot a phase diagram illustrating, for $N = 300$, and several values of $s/m$ and $m/N$ between $0$ and $.7$, the success rate of $\ell_1$-minimization in exactly recovering Legendre $s$-sparse polynomials $g(x) = \sum_{k=0}^{N-1} c_k L_k(x)$.  
The results, illustrated in Figure $2$, show a sharp transition between uniform
recovery (in black) and no recovery whatsoever (white).  This transition curve is similar to the phase transition curves obtained for other compressive sensing matrix ensembles, e.g.\ the random partial discrete Fourier matrix or the Gaussian ensemble.  For more details, we refer the reader to \cite{dt}.


\begin{SCfigure}
\centering
{\label{legendre:plot}
    \psfrag{ylabel}{$\frac{s}{m}$}
                \psfrag{xlabel}{$\frac{m}{N}$}}
\includegraphics[width=7cm, height = 6cm]{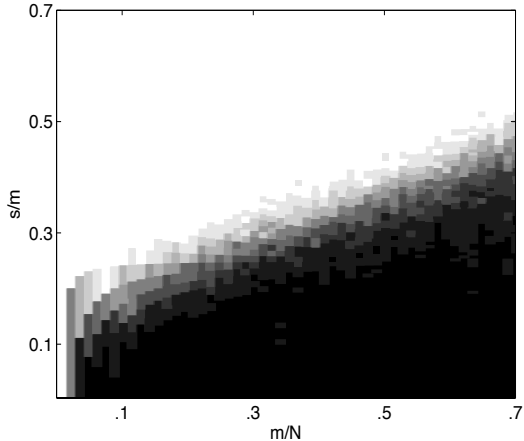}
\caption{Phase diagram illustrating the transition between uniform recovery (black) and no recovery whatsoever (white) of Legendre-sparse polynomials of sparsity level $s$ and using $m$ measurements, as $s$ and $m$ vary over the range $s \leq m \leq N = 300$.  In particular, for each pair $(s/m, m/N)$, we record the rate of success out of $50$ trials of $\ell_1$-minimization in recovering $s$-sparse coefficient vectors with random support over $[N]$ and with i.i.d. standard Gaussian coefficients from $m$ measurements distributed according to the Chebyshev measure.  }
\centering
\end{SCfigure}


\section{Sparse recovery via restricted isometry constants}
 We prove Theorem \ref{uniform:noise} by showing that the preconditioned Legendre matrix ${\cal A}\Phi$ satisfies the  \emph{restricted isometry property} (RIP) \cite{cata06,carota06-1}.  To begin, let us recall the notion of restricted isometry constants for a matrix $\Psi$.

\begin{definition}[Restricted isometry constants]
Let $\Psi \in \C^{m \times N}$. For $s \leq N$, the restricted isometry constant $\delta_s$ 
associated to $\Psi$ is the smallest number $\delta$ for which
\begin{equation}\label{def:RIP}
(1-\delta) \|c\|_2^2 \leq \|\Psi c\|_2^2 \leq (1+\delta) \|c\|_2^2
\end{equation}
for all $s$-sparse vectors $c \in \C^N$.
\end{definition}

Informally, the matrix $\Psi$ is said to have the restricted isometry property if
$\delta_s$ is small for $s$ reasonably large compared to $m$.
For matrices satisfying the restricted isometry property, the following $\ell_1$-recovery results can be shown
\cite{carota06-1,ca08,fola09,fo09}.

\begin{theorem}[Sparse recovery for RIP-matrices]
\label{thm:l1:stable} 
Let $\Psi \in \C^{m \times N}$. Assume
that its restricted isometry constant $\delta_{2s}$ 
satisfies
\begin{equation}
\label{RIP:const}
\delta_{2s} < 3/(4 + \sqrt{6}) \approx 0.4652.
\end{equation}
Let $x \in \C^N$ and assume noisy measurements $y = \Psi x + \eta$ are given with $\|\eta\|_2 \leq \varepsilon$. Let $x^\#$ be the minimizer of 
\begin{align}\label{l1eps:prog}
\arg \min_{z \in \C^N} \quad \| z \|_1 \mbox{ subject to } \quad \|\Psi z - y \|_2 \leq \varepsilon.
\end{align}
Then
\begin{align}
\label{l2noise}
\|x - x^\#\|_2 \leq  C_1 \frac{\sigma_s(x)_1}{\sqrt{s}} + C_2 \varepsilon,
\end{align}
and
\begin{align}
\label{l1noise}
\|x - x^\#\|_1 \leq D_1 \sigma_s(x)_1 + D_2 \sqrt{s} \varepsilon.
\end{align}
The constants $C_1, D_1, C_2, D_2 >0$ depend only on $\delta_{2s}$.
In particular, if $x$ is $s$-sparse then reconstruction is exact, $x^\# = x$.
\end{theorem}
The constant in \eqref{RIP:const} is 
the result of several refinements. 
Cand{\`e}s provided the value
$\sqrt{2}-1$ in \cite{ca08}, Foucart and Lai the value $0.45$ in \cite{fola09}, while
the version in \eqref{RIP:const} was shown in \cite{fo09}. 
The proof of \eqref{l2noise} can be found in \cite{ca08}.  The $\ell_1$-error bound \eqref{l1noise} is straightforward from these calculations, but does not seem to appear explicitly in the literature.  

\medskip

So far, all good constructions of matrices with the restricted isometry property use randomness. 
The RIP constant for a matrix whose entries are (properly normalized) independent and identically distributed Gaussian or Bernoulli random variables satisfies $\delta_{s} \leq \delta$ with probability
at least $1- e^{-c_1(\delta) m}$ provided
\begin{equation}
m \geq  c_2(\delta) s \log(N/s); 
\label{s}
\end{equation}
see for example \cite{badadewa08,cata06,rascva08,ra09-1}. To be more precise, it can be shown that 
$c_1(\delta) = C_1 \delta^{2}$ and $c_2(\delta) = C_2 \delta^{-2}$. 
Lower bounds for Gelfand widths of $\ell_1$-balls show that the bound \eqref{s} is optimal \cite{do06-2,codade09,foparaul10}. 

 If one allows for slightly more measurements than the optimal number \eqref{s}, the restricted isometry property also holds for a rich class of \emph{structured} random matrices; the structure of these matrices allows for fast matrix-vector multiplication, which accelerates the speed of reconstruction procedures such as $\ell_1$ minimization.
A quite general class of structured random matrices are those associated to \emph{bounded orthonormal systems}.  This concept is introduced in \cite{ra09-1}, although it is already contained somewhat implicitly
in \cite{cata06,ruve08} for discrete systems.  Let ${\cal{D}}$ be a measurable space -- for instance, a measurable subset of
$\R^d$ -- endowed with a probability measure $\nu$. Further, let $\{ \psi_j$, $j \in [N] \}$, 
be an orthonormal system of (real or complex-valued) functions on ${\cal D}$, i.e.,
\begin{equation}
\int_{\cal D} \psi_j(x) \overline{\psi_k(x)} d\nu(x) = \delta_{j,k}, \quad k,j \in [N].
\end{equation}
If this orthonormal system is uniformly bounded,
\begin{equation}
\label{Linf_bound}
\sup_{j \in [N]} \|\psi_j\|_\infty = \sup_{j \in [N]} \sup_{x \in {\cal D}} |\psi_j(x)| \leq K
\end{equation}
for some constant $K\geq 1$, we call systems $\{\psi_j\}$ satisfying
this condition \emph{bounded orthonormal systems}.

\begin{theorem}[RIP for bounded orthonormal systems]
\label{thm:BOS:RIP} 
Consider 
the matrix $\Psi \in \C^{m \times N}$ with entries
\begin{equation}\label{def:Phi:matrix}
\Psi_{\ell,k} = \psi_k(x_\ell), \quad \ell \in [m], k \in [N],
\end{equation}
formed by i.i.d.\ samples $x_\ell$ drawn from the orthogonalization measure $\nu$
associated to the bounded
orthonormal system $\{ \psi_j$, $j \in [N] \}$ having uniform bound $K\geq 1$ in \eqref{Linf_bound}. 
If 
\begin{equation}\label{BOS:RIP:cond}
m \geq C\delta^{-2} K^2 s \log^3(s) \log(N),
\end{equation}
then with probability at least 
$1-N^{-\gamma \log^3(s)},$ 
the restricted isometry constant $\delta_s$ of $\frac{1}{\sqrt{m}} \Psi$ satisfies $\delta_s \leq \delta$. The constants $C,\gamma>0$ are universal.
\end{theorem}

We note that condition \eqref{BOS:RIP:cond} is stated slightly different in \cite{ra09-1}, namely as
\[
\frac{m}{\log(m)} \geq C\delta^{-2} K^2 s \log^2(s) \log(N).
\]
However, it is easily seen that \eqref{BOS:RIP:cond} implies this condition (after possibly adjusting constants). 
Note also  that \eqref{BOS:RIP:cond} is implied by the simpler condition
\[
m \geq C K^2 \delta^{-2} s \log^4(N).
\]

\medskip

An important special case of a bounded orthonormal system is the random partial Fourier matrix, which is formed by choosing a random subset of $m$ rows from the $N \times N$ discrete Fourier matrix.   The continuous analog of this system is the matrix associated to the trigonometric polynomial basis 
$\{ x \mapsto e^{2\pi i n x } , \quad n = 0, \hdots, N-1 \}$ evaluated at $m$ sample points chosen 
independently from the uniform measure on $[0,1]$.  Note that the trigonometric system has 
corresponding optimal uniform bound $K=1$.  Another example is the matrix associated to the 
Chebyshev polynomial system evaluated at sample points chosen independently 
from the corresponding orthogonalization measure, the Chebyshev measure. In this case, $K = \sqrt{2}$.  

\section{Proof of Theorem \ref{uniform:noise}}

As a first approach towards recovering Legendre-sparse polynomials from random samples, one may try to apply Theorem \ref{thm:BOS:RIP} directly, selecting the sampling points $\{ x_j, j \in [m] \}$, independently from the normalized Lebesgue measure on $[-1,1]$, the orthogonalization measure for the Legendre polynomials.  However, as shown in \cite{szego}, the $L^{\infty}$-norms of the Legendre polynomials grow according to $\| L_n \|_{\infty} = | L_n (1) | = |L_n(-1) | = (2n + 1)^{1/2}$.  
Applying $K = \| L_{N-1} \|_{\infty} = (2N-1)^{1/2}$ in Theorem \ref{thm:BOS:RIP} produces a required number of samples
\[
m \asymp N \delta^{-2} s \log^3(s) \log(N). 
\]
Of course, this bound is completely useless, because the required number of samples
is now larger than $N$ -- an almost trivial estimate. Therefore, in order to deduce
sparse recovery results for the Legendre polynomials, we must take a different approach.

Despite growing unboundedly with increasing degree at the endpoints $+1$ and $-1$,  an 
important characteristic of the Legendre polynomials is that they are 
all bounded by the same envelope function.
The following result \cite[Theorem 7.3.3]{szego}, gives a precise estimate 
for this bound.
 
\begin{lemma}
\label{thm:growth}
For all $n \geq 1$ and for all $x \in [-1,1]$,
\[
\label{uniform_bound}
(1 - x^2)^{1/4} | \thinspace L_n(x) \thinspace | < 2\pi^{-1/2}\Big( 1 + \frac{1}{2n} \Big)^{1/2}, \hspace{5mm} -1 \leq x \leq 1;
\]
here, the constant $2 \pi^{-1/2}$ cannot be replaced by a smaller one. 
\end{lemma}

\begin{proof}[Proof of Theorem \ref{uniform:noise}]

In light of Lemma \ref{thm:growth}, we apply a preconditioning 
technique to transform the Legendre polynomial system into a bounded orthonormal system.  Consider the functions 
\begin{equation}\label{reweight}
Q_n(x) = (\pi/2)^{1/2} (1 - x^2)^{1/4} L_n(x).
\end{equation}
The matrix $\Psi$ with entries $\Psi_{j,n}=Q_{n-1}(x_j)$ may be written as $\Psi = {\cal A} \Phi$ where ${\cal A}$ is the diagonal matrix with entries $a_{j,j} =  (\pi/2)^{1/2} (1 - x_j^2)^{1/4}$ as in Theorem \ref{thm:growth}, and $\Phi \in \R^{m \times N}$ is the Legendre matrix with entries $\Phi_{j,n} = L_{n-1}(x_j)$.  By Lemma \ref{thm:growth}, the system $\{ Q_n \}$ is uniformly bounded on $[-1,1]$ and satisfies the bound $\| Q_n \|_{\infty} \leq \sqrt{2 + \frac{1}{n}} \leq \sqrt{3}$.  Due to the orthonormality of the Legendre system with respect to the normalized Lebesgue measure on $[-1,1]$, the $Q_n$ are orthonormal with respect to the Chebyshev probability measure $d \nu(x) = \pi^{-1}(1 - x^2)^{-1/2} dx$ on $[-1,1]$:
\begin{eqnarray}
\label{ortho}
\int_{-1}^1 \pi^{-1}Q_n(x) Q_{k}(x) (1-x^2)^{-1/2} dx &=& \frac{1}{2}\int_{-1}^1 L_n(x) L_{k}(x) dx  = \delta_{n,k}. \nonumber
\end{eqnarray}

Therefore, the $\{ Q_n \}$ form a bounded orthonormal system in the sense of Theorem \ref{thm:BOS:RIP} with uniform bound $K = \sqrt{3}$.  By Theorem \ref{thm:BOS:RIP}, the renormalized matrix $\frac{1}{\sqrt{m}} \Psi$ has the restricted isometry property with constant $\delta_s \leq \delta$ with high probability once $m \geq C\delta^{-2}s \log^4(N)$.  We then apply Theorem \ref{thm:l1:stable} to the noisy samples $\frac{1}{\sqrt{m}} {\cal A}y$ where $y = \big( g(x_1) + \eta_1, ... , g(x_m) + \eta_m \big)$ and observe that $\| {\cal A} \eta \|_{\infty} \leq \varepsilon$ implies $\frac{1}{\sqrt{m}} \| {\cal A} \eta \|_2 \leq \varepsilon$.  This gives Theorem \ref{uniform:noise}.

\end{proof}




\section{Universality of the Chebyshev measure}
The Legendre polynomials are orthonormal with respect to the uniform measure on $[-1,1]$; we may instead consider an arbitrary weight function $v$ on $[-1,1]$, and the polynomials $\{p_n\}$ that are orthonormal with respect to $v$.  
Subject to a mild continuity condition on $v$, a result similar to Lemma \ref{thm:growth} concerning the uniform growth of ${p_n}$ still holds, and the sparse recovery results of Theorem \ref{uniform:noise} extend to this more general scenario.  In all cases, the sampling points are chosen according to the Chebyshev measure. 

Let us recall the following general bound, see e.g.\ Theorem 12.1.4 in Szeg{\"o} \cite{szego}. 

\begin{theorem} 
\label{Lip:Dini}
Let $v$ be a weight function on $[-1,1]$ and set
$f_v(\theta) = v(\cos \theta)|\sin(\theta)|$. Suppose that
$f_v$ satisfies the Lipschitz-Dini condition, that is,
\begin{equation}\label{LipschitzDini}
|f_v(\theta+\delta) - f_v(\theta)| \leq L |\log(1/\delta)|^{-1-\lambda}, \quad \mbox{ for all } \theta \in [0,2\pi), \delta > 0,
\end{equation}
for some constants $L,\lambda > 0$. 
 Let $\{p_n, n \in \N_0\}$, be the associated orthonormal polynomial system. Then
\begin{equation}\label{weight:estimate}
(1-x^2)^{1/4} v(x)^{1/2} |p_n(x)| \leq C_v \quad \mbox{ for all } n\in \N, x \in [-1,1].
\end{equation}
The constant $C_v$ depends only on the weight function $v$. 
\end{theorem}
The Lipschitz-Dini condition \eqref{LipschitzDini} is satisfied for a range
of Jacobi polynomials $p_n = p_n^{(\alpha,\beta)}$, $n \geq 0,$ $\alpha,\beta \geq -1/2$, 
which are orthogonal
with respect to the weight function $v(x) = (1-x)^\alpha(1+x)^\beta$. 
The Legendre polynomials are a special case of the Jacobi polynomials corresponding to $\alpha = \beta = 0$; more generally, the case $\alpha = \beta$ correspond to the ultraspherical polynomials. The Chebyshev polynomials are another important special case of ultraspherical polynomials, corresponding to parameters $\alpha = \beta = -1/2$, and Chebyshev measure.

For any orthonormal polynomial system satisfying a bound of the form \eqref{weight:estimate}, the following RIP-estimate applies.
\begin{theorem}
\label{RIP:precondition}
Consider a positive weight function $v$ on $[-1,1]$ satisfying the conditions of Theorem \ref{Lip:Dini}, and consider the orthonormal polynomial system $\{ p_n \}$ with respect to the probability measure $d\nu(x) = c\,v(x) dx$ on $[-1,1]$ where 
$c^{-1} = \int_{-1}^1 v(x)dx$.  

Suppose that $m$ sampling points $(x_1, \hdots, x_m)$ are drawn independently at random from the Chebyshev  
measure, and consider the $m \times N$ composite matrix $\Psi = {\cal A} \Phi$, where $\Phi$ is the matrix with entries $\Phi_{j,n} = p_{n-1}(x_j)$, and ${\cal A}$ is the diagonal matrix with entries $a_{j,j} = (c\pi)^{1/2} (1-x_j^2)^{1/4} v(x_j)^{1/2}$.   Assume that 
\begin{equation}\label{RIP:prec:cond}
m \geq C \delta^{-2} s \log^3(s) \log(N).       
\end{equation}
Then with probability at least $1-N^{-\gamma \log^3(s)}$ the restricted isometry
constant of the composite matrix $\frac{1}{\sqrt{m}} \Psi = \frac{1}{\sqrt{m}} {\cal A} \Phi$ satisfies $\delta_s \leq \delta$. The constant $C$ depends only on $v$, and the constant $\gamma > 0$ is universal.
\end{theorem}

\begin{proof}[Proof of Theorem \ref{RIP:precondition}]
Observe that $\Psi_{j,n} = Q_{n-1}(x_j)$, where 
$$Q_n(x) =(c \pi)^{1/2} (1-x^2)^{1/4} v(x)^{1/2} p_n(x).$$
Following Theorem \ref{Lip:Dini},  the system $\{ Q_n \}$ is uniformly bounded on $[-1,1]$ and satisfies the bound $\| Q_n \|_{\infty} \leq (c\pi)^{-1/2} C_v$; moreover, due to the orthonormality of the polynomials $\{ p_n \}$ with respect to the measure $d\nu(x)$,  the $\{Q_n\}$ are orthonormal with respect to the Chebyshev measure:
\begin{eqnarray}
\label{ortho_nu}
\int_{-1}^1 \pi^{-1}Q_n(x) Q_{k}(x) (1-x^2)^{-1/2} dx &=& \int_{-1}^1 c p_n(x) p_{k}(x) v(x)dx  = \delta_{n,k}. 
\end{eqnarray}
Therefore, the $\{ Q_n \}$ form a bounded orthonormal system with associated matrix $\Psi$ as in Theorem \ref{RIP:precondition} formed from samples $\{ x_j \}$ drawn from the Chebyshev distribution.  Theorem \ref{thm:BOS:RIP} implies that the renormalized composite matrix $\frac{1}{\sqrt{m}}\Psi$ has the restricted isometry property as stated.
\end{proof}

\begin{corollary}
\label{poly:noise}
Consider an orthonormal polynomial system $\{ p_n \}$ associated to a measure $v$ satisfying the conditions of Theorem \ref{Lip:Dini}.
Let $N,m,s \in \N$ satisfy the conditions of Theorem \ref{RIP:precondition}, and 
consider the matrix $\Psi = {\cal{A}} \Phi$ as defined there.  

Then with probability exceeding
$1-N^{-\gamma \log^3(s)}$ the following 
holds for all polynomials $g(x) = \sum_{k=0}^{N-1} c_k p_k(x)$.
If noisy sample values $y = \big( g(x_1) + \eta_1, \hdots, g(x_m) + \eta_m \big) = \Phi c + \eta$ are observed, and $\|\eta\|_{\infty} \leq \varepsilon$, then the coefficient vector $c = (c_0, c_1, \hdots, c_{N-1})$ is recoverable to within a factor of its best $s$-term approximation error  
and to a factor of the noise level by solving the inequality-constrained $\ell_1$-minimization problem
\begin{align}\label{relaxed2}
c^{\#} = \arg \min_{z \in \R^N} \| z \|_1 \quad \mbox{ subject to } \quad \| {\cal A} \Phi z - {\cal A} y \|_2 \leq \sqrt{m}\varepsilon.
\end{align}
Precisely, 
$$
\| c -c^{\#} \|_{2} \leq \frac{C_1 \sigma_s(c)_1}{\sqrt{s}} + D_1\varepsilon,
$$
and
\begin{equation}\label{l1:approx3}
\|c-c^{\#}\|_{1} \leq C_2 \sigma_s(c)_1 + D_2 \sqrt{s} \varepsilon.
\end{equation}
The constants $C_1,C_2,D_1,D_2$ and $\gamma$ are universal.
\end{corollary}



As a byproduct of Theorem \ref{RIP:precondition}, we also obtain condition number estimates
for preconditioned orthogonal polynomial matrices that should be of interest on their own, and improve on the results in \cite{grpora07}.   Theorem \ref{RIP:precondition} implies that all submatrices of a preconditioned random orthogonal polynomial matrix 
$\frac{1}{\sqrt{m}}\Psi = \frac{1}{\sqrt{m}} {\cal A} \Phi \in \R^{m \times N}$ with at most $s$ columns are simultaneously well-conditioned, provided \eqref{RIP:prec:cond} holds. If one is only interested in a particular subset of $s$ columns, i.e., a particular subset of $s$ orthogonal polynomials, the number of measurements in \eqref{RIP:prec:cond} can be reduced to 
\begin{equation}\label{m:est}
m \geq C s \log(s);
\end{equation}
see Theorem 7.3 in \cite{ra09-1} for more details. 

\medskip

\noindent
{\bf Stability with respect to the sampling measure.}

The requirement that sampling points $x_j$ are drawn from the Chebyshev measure in the previous theorems can be relaxed somewhat.    In particular, suppose that the sampling points $x_j$ are drawn  not from the Chebyshev measure, but from a more general probability measure $d\nu(x) = \rho(x) dx$ on $[-1,1]$ with 
$\rho(x) \geq c' (1-x^2)^{-1/2}$ (and $\int_{-1}^1 \rho(x) dx = 1$).  Now assume a weight function $v$ satisfying the Lipschitz-Dini condition \eqref{LipschitzDini} and the associated orthonormal polynomials $p_n(x)$ are given. Then, by Theorem \ref{Lip:Dini} the functions
\begin{equation}\label{Qn:def}
Q_n(x) =(c \pi)^{1/2} \rho(x)^{-1/2} v(x)^{1/2} p_n(x)
\end{equation}
form a bounded orthonormal system with respect to the probability measure $\tilde{c} \rho(x) v(x) dx$. Therefore, all previous
arguments are again applicable. 
We note, however, that taking $\rho(x) dx$ to be the Chebyshev measure produces the smallest constant $K$
in the boundedness condition \eqref{Linf_bound} due to normalization reasons.

%

\section{Recovery in infinite-dimensional function spaces}

We can transform the previous results into approximation results on the level of continuous functions.  For simplicity, we restrict the scope of this section to the Legendre basis, although all of our results extend to any orthonormal polynomial system with a Lipschitz-Dini weight function, as well as to the trigonometric system, for which related results have not been
worked out yet, either.

We introduce the following weighted norm on continuous functions in $[-1,1]$:
\[
\|f\|_{\infty,w} := \sup_{x \in [-1,1]} |f(x)| w(x), \quad w(x) = \sqrt{\frac{\pi}{2}} (1-x^2)^{1/4}.
\]
Further, we define
\begin{equation}
\label{weight:norm}
\sigma_{N,s}(f)_{\infty, w} := \inf_{c \in \R^N} \left\{ \sigma_s(c)_{1} + \sqrt{s} \|f- \sum_{k=0}^{N-1} c_k L_k \|_{\infty, w}\right\}.
\end{equation}
The above quantity involves the best $s$-term approximation error of $c$, as well as the ability of Legendre coefficients 
$c \in \R^{N}$ to approximate the given function $f$ in the $L_\infty$-norm. In some sense, it provides a mixed linear and
nonlinear approximation error. The $c$ which ``balances'' both error terms determines $\sigma_{N,s}(f)_\infty$.
The factor $\sqrt{s}$ scaling the ``linear approximation part'' may seem to lead to non-optimal estimates at first sight, but 
later on, the strategy will actually be to choose $N$ in dependence of $s$ such that $\sigma_{N,s}(f)_\infty$ becomes
of the same order as $\sigma_s(c)_1$. In any case, we 
note the (suboptimal) estimate
\[
\sigma_{N,s}(f)_{\infty, w} \leq \sqrt{s}\, \rho_{N,s}(f)_{\infty,w},
\]
where
\[
\rho_{N,s}(f)_{\infty,w} := \inf_{c \in \R^N, \|c\|_0 \leq s} \|f- \sum_{k=0}^{N-1} c_k L_k \|_{\infty, w}.
\]

Our aim is to obtain
a good approximation to a continuous function $f$ from $m$ sample values, and to compare the approximation
error with $\sigma_{N,s}(f)_{\infty,w}$. We have



\begin{proposition}\label{thm:function:approx} Let $N,m,s$ be given with
\[
m \geq C s \log^3(s) \log(N).
\]
Then there exist sampling points $x_1,\hdots,x_m$ (i.e., chosen i.i.d.\ from the Chebyshev measure) and an efficient 
reconstruction procedure (i.e., $\ell_1$-minimization), such that for
any continuous function $f$ with associated error $\sigma_{N,s}(f)_{\infty, w}$, the
polynomial $P$ of degree at most $N$ 
reconstructed from $f(x_1),\hdots,f(x_m)$ satisfies
\[
\|f-P\|_{\infty,w} \leq C' \sigma_{N,s}(f)_{\infty, w}.
\] 
The constants $C, C'>0$ are universal.
\end{proposition}

The quantity $\sigma_{N,s}(f)_{\infty,w}$ involves 
the two numbers $N$ and $s$. 
We now describe how $N$ can be chosen in dependence on $s$, reducing the number of parameters to one. 
We illustrate this strategy below in a more concrete
situation. To describe the setup we introduce analogues of the Wiener algebra in the Legendre polynomial setting. 
Let $c(f)$ with entries
\[
c_k(f) = \frac{1}{2} \int_{-1}^1 f(x) L_k(x) dx,\quad k \in \N_0, 
\]
denote the vector of Fourier-Legendre coefficients of $f$. Then we define
\[
A_p := \{ f \in C[-1,1], \|c(f)\|_p < \infty \},\quad 0<p\leq 1,
\]
with quasi-norm $\|f\|_{A_p} := \|c(f)\|_p$. The use of the $p$-norm is motivated by the 
Stechkin estimate \eqref{Stechkin} below, which
tells us that elements in $\ell_p$ can be considered compressible. 
Since $\|L_k w\|_{\infty} \leq \sqrt{3}$ it follows that 
\[
f(x) w(x) = \sum_{k \in \N_0} c_k(f) L_k(x) w(x)
\] 
converges uniformly for  $f \in A_1$, so that $fw \in C[-1,1]$, and
$\|f\|_{\infty,w} \leq \sqrt{3}\|f\|_{A_1}$. Since $\|f\|_{A_1} \leq \|f\|_{A_p}$ for $0 <p\leq 1$ this holds also for 
$f \in A_p$, $0<p\leq 1$.  Now we introduce
\[
\sigma_s(f)_{A_1} := \inf_{c \in \ell_2(\N_0), \|c\|_0 \leq s} \|f - \sum_{k} c_k L_k\|_{A_1} = \sigma_s( c(f) )_1.
\]
By Stechkin's estimate \eqref{Stechkin} (which is also valid in infinite dimensions) we
have, for $0< q < 1$,
\begin{equation}\label{Stechkin:rate}
\sigma_s(f)_{A_1} \leq s^{1-1/q} \|f\|_{A_q}.
\end{equation}
Our goal is to realize this approximation rate for $f \in A_q$ when only
sample values of $f$ are given. Additionally, the number of samples should be close (up to $\log$-factors) to the number $s$ of degrees of freedom of the reconstructed function. Unfortunately, for this task we have to at least know roughly a finite set $[N]$ containing  the Fourier-Legendre coefficients of 
a good $s$-sparse approximation
of $f$. In order to deal with this problem, we introduce, for $\alpha >0$, a weighted Wiener type space $A_{1,\alpha}$, containing the functions $f \in C[-1,1]$ with finite norm
\[
\|f\|_{A_{1,\alpha}} := \sum_{k \in \N_0} (1+k)^\alpha |c_k(f)|.
\]
One should imagine $\alpha\ll 1$ very small, so that $f \in A_{1,\alpha}$ 
does not impose a severe restriction on $f$, compared to $f \in A_q$. Then instead of
$f \in A_q$ we make the slightly stronger requirement $f \in A_q \cap A_{1,\alpha}$, $0<q<1$. The next theorem states that under such assumptions, the optimal 
rate \eqref{Stechkin:rate} can be realized when only a small number of sample values of $f$ are available.

\begin{theorem}\label{thm:Ap} Let $0<q<1$, $\alpha >0$, and $m,s \in \N$ be given such that
\begin{equation}\label{ms:rel}
m \geq C \alpha^{-1}\left(\frac{1}{q}-\frac{1}{2}\right) s \log^4(s).
\end{equation}
Then there exist sampling points $x_1,\hdots,x_m \in [-1,1]$ (i.e., random Chebyshev points) such that for every 
$f \in A_q \cap A_{1,\alpha}$ a polynomial $P$ of degree at most 
$N = \lceil s^{(1/q-1/2)/\alpha}\rceil$
can be reconstructed from the sample values $f(x_1),\hdots,f(x_m)$ such that
\begin{equation}\label{recovery:rate}
\frac{1}{\sqrt{3}}\|f - P\|_{\infty,w} \leq \|f-P\|_{A_1} \leq C(\|f\|_{A_q} + \|f\|_{A_{1,\alpha}}) s^{1-1/q}.
\end{equation}
\end{theorem}
Note that up to $\log$-factors the number of required samples is of the
order of the number $s$ of degrees of freedom (the sparsity) 
allowed in the estimate \eqref{Stechkin}, and the reconstruction 
error \eqref{recovery:rate} satisfies the same rate. Clearly $\ell_1$-minimization or greedy alternatives can be used for reconstruction. This result 
may be considered
as an extension of the theory of compressive sensing to infinite dimensions (although all the
key tools are actually finite dimensional).

\subsection{Proof of Proposition \ref{thm:function:approx}}

Let $P_{opt} = \sum_{k=0}^{N-1} c_{k,opt} L_k$ denote the polynomial of degree at most $N-1$ whose coefficient vector $c_{opt}$ realizes the approximation error $\sigma_{N,s}(f)_{\infty, w}$, as defined in \eqref{weight:norm}. The samples $f(x_1), \hdots, f(x_m)$ can be seen as noise corrupted samples of $P_{opt}$, that is, $f(x_{\ell}) = P_{opt}(x_{\ell}) + \eta_{\ell},$ and $| \eta_{\ell} | w(x_{\ell}) \leq \| f - P_{opt} \|_{\infty, w} := \varepsilon$.  The preconditioned system reads then $f(x_{\ell}) w(x_{\ell}) = \sum_{k=0}^{N-1} c_{k,opt} L_k(x_{\ell}) w(x_{\ell}) + \varepsilon_{\ell}$, with $| \varepsilon_{\ell} | \leq \varepsilon$.  According to Theorem \ref{thm:BOS:RIP} and Theorem \ref{thm:growth}, the matrix $\frac{1}{\sqrt{m}}\Psi$ consisting of entries $\Psi_{\ell,k} = w(x_{\ell})L_{k-1}(x_\ell)$ satisfies the RIP with high probability, provided the stated condition on the minimal number of samples holds.
Due to Theorem \ref{thm:l1:stable}, an application of noise-aware $\ell_1$-minimization \eqref{l1eps:prog} to $y = (f(x_\ell)w(x_{\ell}))_{\ell=1}^m$ with $\varepsilon$ replaced
by $\sqrt{m}\varepsilon$ yields a coefficient vector $c$ satisfying $\| c - c_{opt} \|_1 \leq C_1 \sigma_s(c_{opt})_1 + C_2 \sqrt{s} \varepsilon$.   We denote the polynomial corresponding to this coefficient vector by $P(x) = \sum_{k=0}^{N-1} c_k L_k(x)$. Then 
\begin{align}
\| f - P \|_{\infty, w} &\leq
 \| f - P_{opt} \|_{\infty, w} + \| P_{opt} - P \|_{\infty, w} 
\leq \frac{\sigma_{N,s}(f)_{\infty, w}}{\sqrt{s}} + \sqrt{3} \| c - c_{opt} \|_1 \nonumber \\
&\leq \frac{\sigma_{N,s}(f)_{\infty, w}}{\sqrt{s}} + \sqrt{3} \Big[ C_1\sigma_s(c_{opt})_1 +  C_2 \sqrt{s}  \| f - P_{opt} \|_{\infty, w}\Big] 
\leq C \sigma_{N,s}(f)_{\infty, w}.\nonumber
\end{align}
This completes the proof.

The attentive reader may have noticed that our recovery method, noise-aware 
$\ell_1$-minimization \eqref{l1eps:prog}, requires knowledge of $\sigma_{N,s}(f)$, see also
Remark \ref{rem22}(c). 
One may remove this drawback by considering CoSaMP \cite{netr08} or
Iterative Hard Thresholding \cite{blda09} instead.  The required 
error estimate in $\ell_1$ follows from the $\ell_2$-stability results for these algorithms
in \cite{blda09,netr08}, as both algorithms produce a $2s$-sparse vector, 
see \cite[p.\ 87]{b09} for details.   

\subsection{Proof of Theorem \ref{thm:Ap}}

Let $f \in A_q \cap A_{1,\alpha}$ with Fourier Legendre coefficients $c_k(f)$.
Let $N > s$ be a number to be chosen later and introduce the truncated Legendre expansion
\[
f_N(x) = \sum_{k=0}^{N-1} c_k(f) L_k(x),
\]
which has truncated Fourier-Legendre coefficient vector $c^{(N)}$ with entries
$c^{(N)}_k =c_k(f)$ if $k<N$ and $c^{(N)}_k = 0$ otherwise. Clearly, 
$
\|c^{(N)}\|_q \leq \|c(f)\|_q = \|f\|_{A_q}.
$
Further note that
\begin{align}
\frac{1}{\sqrt{3}}\|f-f_N\|_{\infty,w} & \leq \|f-f_N\|_{A_1} = \|c-c^{(N)}\|_1 
= \sum_{k=N}^\infty |c_k(f)| \leq N^{-\alpha} \sum_{k=N}^{\infty} (1+k)^{\alpha} |c_k(f)|\notag\\
& \leq N^{-\alpha} \|f\|_{A_{1,\alpha}}.\notag
\end{align}
Now we proceed similarly as in the proof of Theorem \ref{thm:function:approx} and treat the samples of $f$ as perturbed
samples of $f_N$, that is $f_N(x_j) = f(x_j) + \eta_j$ with 
$|\eta_j| w(x_j) \leq \|f-f_N\|_{\infty,w} \leq \sqrt{3}N^{-\alpha} \| f \|_{A_{1,\alpha}}$. Then following the same
arguments as in the proof of Theorem \ref{thm:function:approx}, if
\begin{equation}\label{msN:choice}
m \geq C s \log^3(s) \log(N),
\end{equation}
we can reconstruct a 
coefficient vector $\widetilde{c}$ from samples $f(x_1),\hdots,f(x_m)$ 
with support contained in $\{0,1,\hdots,N-1\}$ such that 
\begin{align}
\|c^{(N)} - \widetilde{c}\|_1 \leq C_1 \sigma_{s}(c^{(N)})_1 + C_2 \sqrt{s} \|f-f_N\|_{\infty,w}
\leq C_1 s^{1-1/q} \|f\|_{A_q} + C_2 \sqrt{s} N^{-\alpha} \|f\|_{A_{1,\alpha}}.\notag
\end{align}
Here, we applied Stechkin's estimate \eqref{Stechkin}. Therefore,
\begin{align}
\|c-\widetilde{c}\|_1 &\leq \|c-c^{(N)}\|_1 + \|c^{(N)} - \widetilde{c}\|_1
\leq N^{-\alpha} \|f\|_{A_{1,\alpha}} + C_1 s^{1-1/q} \|f\|_{A_q} + C_2 \sqrt{s} N^{-\alpha} \|f\|_{A_{1,\alpha}}\notag\\
&\leq C_1 s^{1-1/q} \|f\|_{A_q} + C_2' \sqrt{s} N^{-\alpha} \|f\|_{A_{1,\alpha}}.\notag
\end{align}
Now we choose 
\begin{align}\label{N:choice}
N = \lceil s^{1/\alpha(1/q-1/2)} \rceil
\end{align}
which yields $\sqrt{s} N^{-\alpha} \leq s^{1-1/q}$. With this choice
\[
\frac{1}{\sqrt{3}}\|f-\tilde{f}_N\|_{\infty,w} \leq \|f-\tilde{f}_N\|_{A_1} = \|c-\widetilde{c}\|_1 \leq C' (\|f\|_{A_q} + \|f\|_{A_{1,\alpha}}) s^{1-1/q}.
\]
Plugging \eqref{N:choice} into \eqref{msN:choice} yields \eqref{ms:rel}, 
and the proof is finished.


\begin{remark}
\emph{Analogous function approximation results can be derived from Theorem \ref{RIP:precondition} for any orthogonal polynomial basis whose weight function satisfies the conditions of Theorem \ref{Lip:Dini}.  The associated norm is $\| f \|_{v, \infty} = \| \sqrt{3} v^{1/2} f w \|_{\infty}.$    For the Chebyshev polynomials, $\| f \|_{v, \infty} = \| f \|_{\infty}$, and the corresponding function approximation results in this case are with respect to the unweighted uniform norm.  }
 
\end{remark}

\subsection*{Acknowledgments}
The authors would like to thank Albert Cohen, Simon Foucart, and Joseph Ward
for valuable discussions on this topic, and are also grateful to Laurent Gosse for helpful comments.   Rachel Ward gratefully awknowledges the partial support of National Science Foundation Postdoctoral Research Fellowship.  Holger Rauhut gratefully acknowledges support by 
the Hausdorff Center for Mathematics and by the WWTF project SPORTS (MA 07-004).
Parts of this manuscript have been written during a stay of the first author at the Laboratoire
Jacques-Louis Lions of 
Universit{\'e} Pierre et Marie Curie in Paris. He greatly acknowledges the warm hospitality of
the institute and especially of his host Albert Cohen.

\bibliography{SparseLegendre}
\bibliographystyle{abbrv}



\end{document}